\documentclass[graybox,envcountsect,envcountsame]{svmult}
\usepackage{etex}

\usepackage{mathptmx}       
\usepackage{helvet}         
\usepackage{courier}        
\usepackage{type1cm}        
\usepackage{makeidx}         
\usepackage{graphicx}        
\usepackage{multicol}        
\usepackage[bottom]{footmisc}

\makeindex             

\usepackage{amstext,amsfonts,amssymb,amscd,amsbsy,amsmath}
\usepackage{tikz-cd}	
\usepackage{tikz}

\usepackage[all]{xy}

\usepackage{enumerate}
\usepackage{caption}
\usepackage{float}

\smartqed

\newcommand{\PP}{\mathbb{P}}

\newcommand{\RR}{\mathbb{R}}
\newcommand{\ZZ}{\mathbb{Z}}

\DeclareMathOperator{\trop}{trop}

\DeclareMathOperator{\Rat}{Rat}
\DeclareMathOperator{\Div}{Div}
\DeclareMathOperator{\PDiv}{PDiv}
\DeclareMathOperator{\ord}{ord}
\DeclareMathOperator{\val}{val}
\DeclareMathOperator{\Aut}{Aut}
\DeclareMathOperator{\Spec}{Spec}

\newcommand{\calC}{\mathcal{C}}
\newcommand{\calH}{\mathcal{H}}

\newcommand{\calM}{\mathcal{M}}
\newcommand{\calMbar}{\overline{\mathcal{M}}}

\begin{document}

\title*{Towards a tropical Hodge bundle}
\author{Bo Lin and Martin Ulirsch}
\institute{Bo Lin \at Department of Mathematics, University of California, Berkeley, Berkeley, CA 94720, \email{linbo@math.berkeley.edu}
\and Martin Ulirsch \at Fields Institute for Research in Mathematical Sciences, University of Toronto, 222 College Street, Toronto, Ontario M5T 3J1 \email{mulirsch@fields.utoronto.ca}}
%
%
\maketitle

\abstract{The moduli space $M_g^{trop}$ of tropical curves of genus $g$ is a generalized cone complex that parametrizes metric vertex-weighted graphs of genus $g$. For each such graph $\Gamma$, the associated canonical linear system $\vert K_\Gamma\vert$ has the structure of a polyhedral complex. In this article we propose a tropical analogue of the Hodge bundle on $M_g^{trop}$ and study its basic combinatorial properties. Our construction is illustrated with explicit computations and examples.}


\section{Introduction} Let $g\geq 2$ and denote by $\calM_g$ the moduli space of smooth algebraic curves of genus $g$. The \emph{Hodge bundle} $\Lambda_g$ is a vector bundle on $\calM_g$ whose fiber over a point $[C]$ in $\calM_g$ is the vector space $H^0(C,\omega_C)$ of holomorphic differentials on $C$. One can think of the total space of $\Lambda_g$  as parametrizing pairs $(C,\omega)$ consisting of a smooth algebraic curve and a differential $\omega$ on $C$. Since for every curve $C$ the canonical linear system $\vert K_C\vert$ can be identified with the projectivization $\PP\big(H^0(C,\omega_C)\big)$, the total space of the projectivization $\calH_g=\PP(\Lambda_g)$ of $\Lambda_g$ parametrizes pairs $(C,D)$ consisting of a smooth algebraic curve $C$ and a canonical divisor $D$ on $C$; it is referred to as the \emph{projective Hodge bundle}. Let $\pi\mathrel{\mathop:}\mathcal{C}_g\rightarrow\calM_g$ be the universal curve on $\calM_g$. We may define $\Lambda_g$ formally as the pushforward $\pi_\ast\omega_g$ of the relative dualizing sheaf $\omega_g$ on $\calC_g$ over $\calM_g$. 

The Hodge bundle is of fundamental importance when describing the geometry of $\calM_g$. For example, its Chern classes, the so-called \emph{$\lambda$-classes}, form an important collection of elements in the tautological ring on $\calM_g$ (see \cite{Vakil_tautologicalring} for an introductory survey). The Hodge bundle admits a natural stratification by prescribing certain pole and zero orders $(m_1,\ldots, m_n)$ such that $m_1+\ldots+m_n=2g-2$ and the study of natural compactifications of these components has recently seen an surge from both the perspective of algebraic geometry as well as from Teichm\"uller theory (see e.g. \cite{BainbridgeChenGendronGrushevskyMoeller_abeliandifferentials}). 

In tropical geometry, the natural analogue of $\calM_g$ is the moduli space $M_g^{trop}$ that parametrizes  isomorphism classes $[\Gamma]$ of stable tropical curves $\Gamma$ of genus $g$. In Section \ref{section_tropMg} below we are going to recall the construction of this moduli space. In particular, we are going to see how this moduli space naturally admits the structure of a generalized cone complex whose cones are in a natural order-reversing one-to-one correspondence with the boundary strata of the Deligne-Mumford compactification $\calMbar_g$ of $\calM_g$ (see \cite{AbramovichCaporasoPayne_tropicalmoduli} as well as Section \ref{section_tropMg} below for details). 

We also refer the reader to \cite{GathmannKerberMarkwig_tropicalfans}, \cite{GathmannMarkwig_tropicalKontsevich},  \cite{Mikhalkin_ICM}, and \cite{Mikhalkin_Gokova} for the theory in genus $g=0$ (with marked points), to  \cite{BrannettiMeloViviani_tropicalTorelli}, \cite{CaporasoViviani_tropicalTorelli}, \cite{Chan_tropicalTorelli}, \cite{ChanMeloViviani_tropicalmoduli}, and \cite{Viviani_tropvscompTorelli} for its connections to the tropical Torelli map, as well as to \cite{AbramovichCaporasoPayne_tropicalmoduli},  \cite{CavalieriHampeMarkwigRanganathan_tropicalHassett}, \cite{CavalieriMarkwigRanganathan_tropicalHurwitz}, and \cite{Ulirsch_tropicalHassett} for connections of $M_g^{trop}$ (and some of its variants) to non-Archimedean analytic geometry and to \cite{Chan_topologyM2n} and \cite{ChanGalatiusPayne_tropicalmoduliII} for an in-depth study of the topology of $M_{g,n}^{trop}$. We, in particular, highlight the two survey papers \cite{Caporaso_tropicalmoduli} and \cite{Caporaso_tropicalmoduliII}.

Let $\Gamma$ be a tropical curve. We denote by $K_\Gamma$ the canonical divisor on $\Gamma$ and by $\Rat(\Gamma)$ the group of piecewise integer linear functions on $\Gamma$ (see Section \ref{section_linearsystems} below for details). In this note we propose tropical analogues of the affine and the projective Hodge bundle, and study their basic combinatorial properties.  

\begin{definition}\label{def_tropHodge}
As a set, the \emph{tropical Hodge bundle $\Lambda_g^{trop}$} is given as
\begin{equation*}
\Lambda_g^{trop}=\big\{([\Gamma], f)\big\vert [\Gamma]\in M_g^{trop} \textrm{ and } f\in\Rat(\Gamma) \textrm{ such that }K_\Gamma+(f)\geq 0\big\} 
\end{equation*}
and the \emph{projective tropical Hodge bundle $\calH_g^{trop}$} is given as
\begin{equation*}
\calH_g^{trop}=\big\{([\Gamma], D)\big\vert [\Gamma]\in M_g^{trop} \textrm{ and } D\in \vert K_\Gamma \vert\big\} \ .
\end{equation*}
\end{definition}

The associations $\big([\Gamma],f\big)\mapsto [\Gamma]$ and $\big([\Gamma],D\big)\mapsto [\Gamma]$ define natural projection maps $\Lambda_g^{trop}\longrightarrow M_g^{trop}$ and $\calH_g^{trop}\longrightarrow M_g^{trop}$, which, in a slight abuse of notation, we denote both by $\pi_g$. 

In \cite{GathmannKerber_RiemannRoch,HaaseMusikerYu_linearsystems,MikhalkinZharkov_tropicalJacobians} the authors describe a structure of a polyhedral complex on the linear system $\vert D\vert$ associated to a divisor $D$ on a tropical curve $\Gamma$; we are going to review this description in Section \ref{section_linearsystems} below. We also, in particular, highlight the first authors \cite{Lin_linearsystems}, where he presents algorithms for computing this polyhedral complex. 
Our main result is the following:

\begin{theorem}\label{thm_main} Let $g\geq 2$. 
\begin{enumerate}[(i)]
\item The tropical Hodge bundle $\Lambda_g^{trop}$ and the projective tropical Hodge bundle $\calH^{trop}_g$ carry the structure of a generalized cone complex.
\item The dimensions of $\Lambda_g^{trop}$ and $\calH_g^{trop}$ are given by
\begin{equation*}
\dim\Lambda_g^{trop}=5g-4\qquad \textrm{ and } \qquad \dim\calH_g^{trop}=5g-5
\end{equation*}
respectively.
\item There is a proper subdivision of $M_g^{trop}$ such that, for all $[\Gamma]$ in the relative interior of a cone in this subdivision, the canonical linear systems 
\begin{equation*}
\vert K_\Gamma\vert=\pi_g^{-1}\big([\Gamma]\big)
\end{equation*}
have the same combinatorial type. 
\end{enumerate}
\end{theorem}

We are going to refer to this subdivision of $M_g^{trop}$ as the \emph{wall-and-chamber} decomposition of $M_g^{trop}$. In general, the generalized cone complexes $\Lambda_g^{trop}$ and $\calH_g^{trop}$ are not equi-dimensional. So Theorem \ref{thm_main} (ii) really states that the dimension of a maximal-dimensional cone in $\Lambda_g^{trop}$ (or $\calH_g^{trop}$) has dimension $5g-4$ (or $5g-5$ respectively). 

As a first example we refer the reader to Figure \ref{figure_lambda2} below, which depicts the face lattice of the tropical Hodge bundle in the case $g=2$.

\begin{figure}[h]
	\centering
	\includegraphics[scale=0.2]{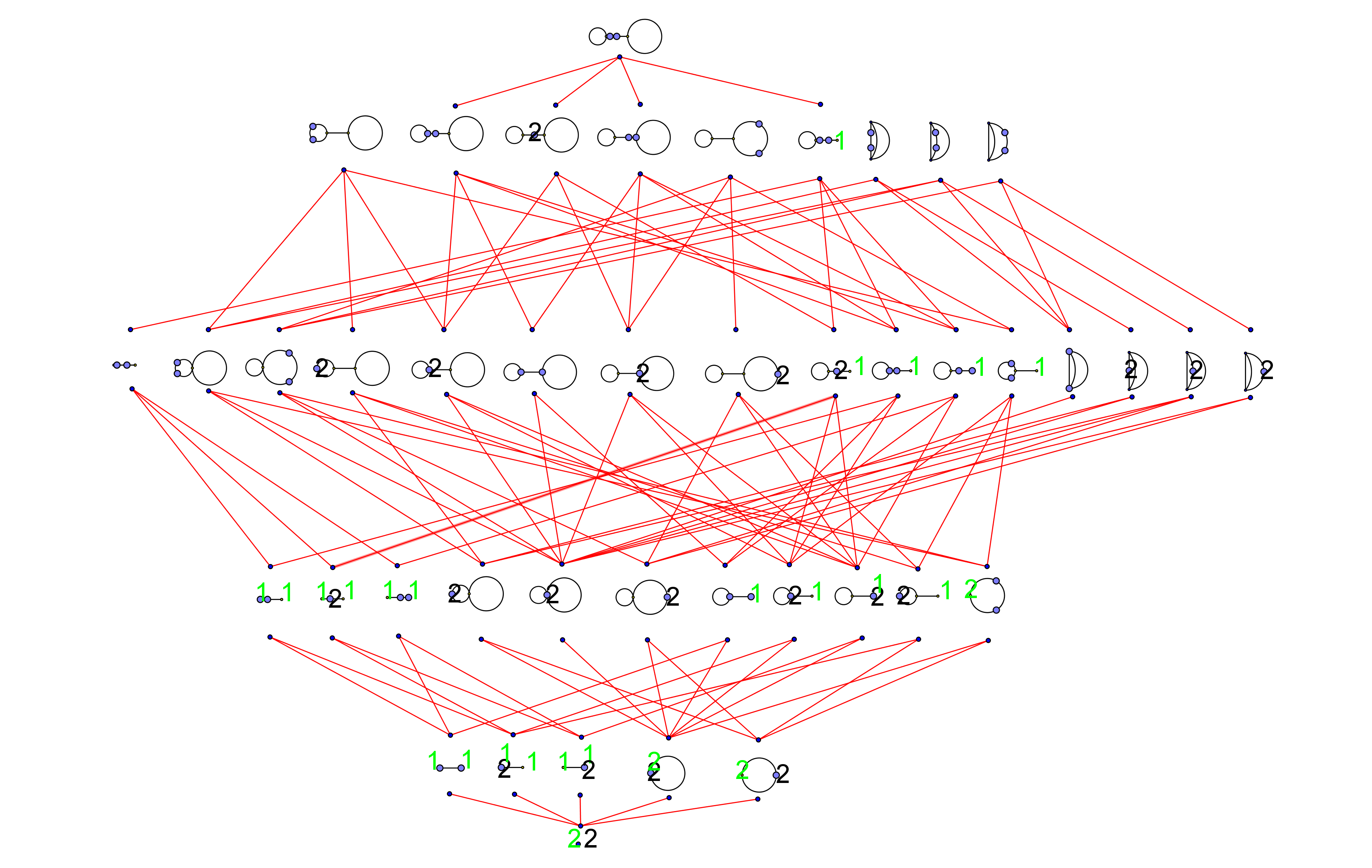}
	\caption{The face lattice of $\calH_2^{trop}$. The numbers in green are the positive $h(v)$ and 
	 the numbers in black denote coefficients greater than $1$ in the divisors.}\label{figure_lambda2}
\end{figure}

Let us give a quick outline of the contents of this contribution. In Section \ref{section_tropMg} we recall the construction of the moduli space $M_g^{trop}$ of stable tropical curves and in Section \ref{section_linearsystems}   the polyhedral structure of linear systems on tropical curves respectively. In Section \ref{section_tropHodge} we prove Theorem \ref{thm_main} by describing the polyhedral structures of both $\Lambda_g^{trop}$ and $\calH_g^{trop}$ simultanously. Section \ref{section_examples} contains a selection of explicit (sometimes partial) calculations of the polyhedral structure of $\calH_g^{trop}$ in some small genus cases. 
Finally, in Section \ref{section_realizability} we describe a natural tropicalization procedure for the projective algebraic Hodge bundle via non-Archimedean analytic geometry and exhibit a natural \emph{realizability problem}. 



\section{Moduli of tropical curves}\label{section_tropMg}

A \emph{tropical curve} is a finite metric graph $\Gamma$ (with a fixed minimal model $G$) together with a genus function $h\mathrel{\mathop:}V(G)\rightarrow\ZZ_{\geq 0}$. The \emph{genus} of $\Gamma$ (or of $G$) is defined to be
\begin{equation*}
g(\Gamma)=g(G)=b_1(G)+\sum_{v\in V(G)}h(v) \,
\end{equation*}
where $b_1(G)$ denotes the Betti number of $G$.
In the above sum one should think of the vertex-weight terms as the contributions of $h(v)$ infinitesimally small loops at every vertex $v$. 
We say a tropical curve $\Gamma$ (or the graph $G$) is \emph{stable}, if for every vertex $v\in V(G)$ we have 
\begin{equation}\label{inequality_weight}
2h(v)-2+\vert v\vert>0 \ ,
\end{equation} 
where $\vert v\vert$ denotes the valence of $G$ at $v$.

\begin{definition}
As a set, the \emph{moduli space $M_g^{trop}$ of stable tropical curves of genus $g$} is given as
\begin{equation*}
M_g^{trop}=\big\{ \textrm{ isomorphism classes } [\Gamma] \textrm{ of stable tropical curves of genus } g \big\} \ .
\end{equation*}
\end{definition}

Let us now recall from \cite{AbramovichCaporasoPayne_tropicalmoduli} the description of $M_g^{trop}$ as a generalized extended cone complex. 

\begin{proposition}[\cite{AbramovichCaporasoPayne_tropicalmoduli} Section 4]
	The moduli space $M_g^{trop}$ carries the structure of a generalized rational polyhedral cone complex that is equi-dimensional of dimension $3g-3$.
\end{proposition}

First, recall that a morphism $\tau\rightarrow\sigma$ between rational polyhedral cones is said to be a \emph{face morphism}, if it induces an isomorphism onto a face of $\sigma$. Note that we explicitly allow the class of face morphisms to include all isomorphisms. A \emph{generalized (rational polyhedral) cone complex} is a topological space $\Sigma$ that arises as a colimit of a finite diagram of face morphisms (see \cite[Section 2]{AbramovichCaporasoPayne_tropicalmoduli} and \cite[Section 3.5]{Ulirsch_functroplogsch} for details). 

In order to understand this structure on $M_g^{trop}$, we observe that it is given as a colimit
\begin{equation*}
M_g^{trop}=\lim_{\rightarrow} \widetilde{M}_G  \ ,
\end{equation*}
of rational polyhedral cones $\widetilde{M}_G$ taken over a category $J_g$. Let us go into some more detail:

\begin{enumerate}
\item The category $J_g$ is defined as follows: 
\begin{itemize}
\item its objects are stable vertex-weighted graphs $(G,h)$ of genus $g$, and
\item its morphisms are generated by \emph{weighted edge contractions} $G\rightarrow G/e$ for an edge $e$ of $G$ as well as by the automorphisms of all $(G,h)$. 
\end{itemize}
Here a \emph{weighted edge contraction} $c\mathrel{\mathop:}G\rightarrow G/e$ is an edge contraction such that for every vertex $v$ in $G/e$ we have
\begin{equation*}
g\big(c^{-1}(v)\big)=h(v) \ .
\end{equation*}
\item Moreover, for every graph $G$ we denote by  
\begin{equation*}
\widetilde{M}_G=\RR_{\geq 0}^{E(G)}
\end{equation*}
the parameter space of all possible edge lengths on $G$. 
\item The association $G\mapsto \widetilde{M}_G$ defines a contravariant functor $J_g\rightarrow \mathbf{RPC}_\ZZ$ from $J_g$ to the category of rational polyhedral cones. It associates to a weighted edge contraction $G\rightarrow G/e$ the embedding of the corresponding face of $\widetilde{M}_G^{trop}$ and to an automorphism of $G$ the automorphism of $\widetilde{M}_G$ that permutes the entries correspondingly. 
\end{enumerate} 
We note hereby that we have a decomposition into locally closed subsets
\begin{equation*}
M_g^{trop}=\bigsqcup_{G}\RR_{>0}^{E(G)}/\Aut (G) \ ,
\end{equation*}
where the disjoint union is taken over the objects in $J_g$, i.e. over all isomorphism classes of stable finite vertex-weighted graphs $G$ of genus $g$. 

\begin{example}[\cite{Chan_tropicalTorelli} Theorem 2.12]
For a $d$-dimensional cone complex $C$, its $f$-vector is defined as $(f_0,f_1,\ldots,f_d)$, where $f_i$ is the number of $i$-dimensional cones in $C$.
The $12$-dimensional moduli space $M_5^{trop}$ has $4555$ cells; its $f$-vector is given by
\begin{equation*}
f(M_5^{trop})=(1,3,11,34,100,239,492,784,1002,926,632,260,71) \ . 
\end{equation*}
\end{example}

\begin{remark}
Earlier approaches, such as \cite{BrannettiMeloViviani_tropicalTorelli}, \cite{Caporaso_tropicalmoduli}, \cite{CaporasoViviani_tropicalTorelli}, \cite{Chan_tropicalTorelli}, \cite{ChanMeloViviani_tropicalmoduli}, and \cite{Viviani_tropvscompTorelli}, used to refer to the structure of a generalized cone complex as a \emph{stacky fan}. Since there is a closely related, but not equivalent, notion of the same name in the theory of toric stacks we prefer to follow the terminology of generalized cone complexes introduced in \cite{AbramovichCaporasoPayne_tropicalmoduli}.
\end{remark}


\section{Linear systems on tropical curves}\label{section_linearsystems}
Let $\Gamma$ be a tropical curve. A \emph{divisor} on $\Gamma$ is a finite formal $\mathbb{Z}$-linear sum 
\begin{equation*}
D=\sum_{i}a_i p_i \ ,
\end{equation*}
over points $p_i$ in $\Gamma$, i.e. $D$ is an element in the free abelian group $\Div(\Gamma)$ on the points of $\Gamma$. The \emph{degree} $\deg(D)$ of a divisor $D=\sum_ia_ip_i$ is defined to be the integer $\sum_i a_i$. We say $D=\sum_i a_ip_i$ is \emph{effective}, if $a_i\geq 0$ for all $i$. 

A \emph{rational function} on $\Gamma$ is a continuous function $f\mathrel{\mathop:}\Gamma\rightarrow \RR$ whose restriction to every edge is a piecewise linear integral affine function. Given a rational function $f$ on $\Gamma$ as above and a point $p\in \Gamma$, the \emph{order} $\ord_p(f)$ of $f$ at $p$ is defined to be the sum of the outgoing slopes of $f$ emanating from $p$. Observe that $\ord_p(f)$ is equal to zero for all but finitely many points $p\in \Gamma$. So we have a map
\begin{equation*}\begin{split}
(.)\mathrel{\mathop:}\Rat(\Gamma)&\longrightarrow \Div(\Gamma)\\
f&\longmapsto (f)=\sum_p\ord_p(f)\cdot p.
\end{split}\end{equation*}

Divisors of the form $(f)$ for a function $f\in \Rat(\Gamma)$ form a subgroup $\PDiv(\Gamma)$ of $\Div(\Gamma)$ and are referred to as the \emph{principal divisors} on $\Gamma$. Two divisors $D$ and $D'$ on $\Gamma$ are said to be \emph{equivalent} (written as $D\sim D'$), if $D-D'\in\PDiv(\Gamma)$, i.e. if there is a rational function $f\in \Rat(\Gamma)$ such that $D+(f)=D'$. Note that the continuity of $f$ implies that $\deg(f)=0$.

Let us now define the main players of this game:
\begin{definition}
Let $D$ be a divisor of degree $n$ on a tropical curve $\Gamma$. 
\begin{enumerate}
\item Denote by $R(D)$ the set
\begin{equation*}
R(D)=\big\{f\in\Rat(\Gamma)\big\vert D+(f)\geq 0\big\} \ .
\end{equation*}
For $f\in R(D)$, the divisor $D+(f)$ is supported in $\deg(D+(f))=\deg(D)=n$ points (counted with multitplicity). We may therefore define:
\begin{equation*}\begin{split}
S(D)=\big\{(f,p_1,\ldots, p_n)&\big\vert f\in\Rat(\Gamma) \textrm{ and } p_1,\ldots, p_n\in\Gamma \\ &\textrm{ such that } D+(f)=p_1+\ldots +p_n\geq 0\big\} \ .
\end{split}\end{equation*}

\item The \emph{linear system} $\vert D\vert$ associated to $D$ is the set
\begin{equation*}
\vert D\vert =\big\{D'\in\Div(\Gamma)\big\vert D'\geq 0 \textrm{ and } D\sim D' \big\} \ . 
\end{equation*}
\end{enumerate}
\end{definition}

Observe that $R(D)=S(D)/S_n$, where the symmetric group $S_n$ acts on $S(D)$ by permutation of the points $p_1,\ldots, p_n$. Moreover, the additive group $\RR=(\RR, +)$ operates on $R(D)$ by adding a constant function and, taking the quotient under this operation, we obtain that 
\begin{equation*}
R(D)/\RR=\vert D\vert \ ,
\end{equation*}
since $(f)=0$ if and only if $f$ is a constant function on $\Gamma$.

The spaces $S(D)$, $R(D)$, and $\vert D\vert$ are known to carry the structure of a polyhedral complex (see e.g. \cite{MikhalkinZharkov_tropicalJacobians} or \cite{GathmannKerber_RiemannRoch}). The following proposition is a more detailed version of \cite[Lemma 1.9]{GathmannKerber_RiemannRoch}.
 
\begin{proposition}\label{prop_polyhedral}
Given a divisor $D$ on a tropical curve $\Gamma$, the space $S(D)$ has the structure of a polyhedral complex. Choose an orientation for each edge $e$ of $\Gamma$, identifying it with the open interval $[0,l(e)]$. Then the cells of $S(D)$ can be described by the following (discrete) data:
\begin{enumerate}[(i)]
\item a partition of $\{p_1,\ldots, p_n\}$ into disjoint subsets $P_e$ and $P_v$ (indexed by $v\in V(G)$ and edges $e\in E(G)$) that tells us on which edge (or at which vertex) every $p_i$ is located,
\item a total order on each $P_e$, and
\item the outgoing slope $m_e\in \mathbb{Z}$ of $f$ at the starting point of $e$
\end{enumerate} 
such that for every vertex $v$ the equality
\begin{equation*}
\# P_v=D(v)+ \sum_{\textrm{outward edges at } v}m_e +\sum_{\textrm{inward edges at } v} -(\# P_e+m_e)  
\end{equation*}
holds. 
Furthermore, this polyhedral structure descends from $S(D)$ to $R(D)=S(D)/S_n$ and $\vert D\vert=R(D)/\mathbb{R}$.
\end{proposition}

\begin{proof}
Set $d_v=\# P_v$ and $d_e=\# P_e$. We claim that the points in a cell of $S(D)$ can be parametrized by the following two types of continuous data:
\begin{itemize}
\item the value $f(v)$ at a vertex $v$, as well as
\item the distance $d(p_i^e)$ of every $p_i^e\in P_e$ from $0\in e=[0,l(e)]$. 
\end{itemize}

The distances $d(p_i^e)$ immediately determine the $p_i$. In order to reconstruct $f$ (if it exists) we write $\sum_{p\in P_e}p=\sum_j d_{e,j} x_j$ for points $0<x_1<\cdots< x_r<l(e)$ on $e$, where the positive integers $d_{e,j}$ indicate the number of $p_i^e$ that are all located at the same point $x_j$. The rational function $f$ is then determined by taking the value $f(v)$ at the origin of every edge $e=[0,l(e)]$ and continuing it piecewise linearly with slope $m_e$ until we hit $x_1$, at which point we change the slope to $d_{e,1}+m_e$ until we hit $x_2$, where we change the slope to $d_{e,2}+d_{e,1} + m_e$, and so on until we hit the vertex $v'$ at the end of $e=[0,l(e)]$. So, by continuity, for every edge we obtain the linear condition 
\begin{equation*}\begin{split}
f(v')&=f(v)+m_ex_1+\sum_{k=1}^{r}\big(m_e +  \sum_{j=1}^{k} d_{e,j}\big) (x_{k+1}-x_k)=\\
&=f(v) + m_el(e) + \sum_{i=1}^{r}d_{e,i}\big(l(e)-x_i\big)
\end{split}
\end{equation*}
on the parameters of a cell in $S(D)$. This, together with the inequalities $0<x_1<\cdots< x_r<l(e)$ determines the polyhedral structure of a cell in $S(D)$. Note that our parameters are still overdetermined in the sense that there may be no rational function $f$ such that $D+(f)=p_1+\ldots+p_n\geq 0$ and which also fulfills all of the above inequalities; in this case we obtain an empty cell. 

The conditions on the cells of $S(D)$ are all discrete and the points within one cell are all parametrized by the distances $d(p_i^{e})\in(0,l(e))$ and the values $f(v)$ subject to these discrete conditions. Therefore $S(D)$ is a polyhedral complex that does not depend on the choice of the orientation of $\Gamma$. 

The action of $S_n$ on every cell is affine linear and therefore the polyhedral structure descends to $R(D)$. Moreover, the additive group $\mathbb{R}$ acts on $R(D)$ by adding a constant to all $f(v)$ and therefore the polyhedral structure also descends to $\vert D\vert$. 
\qed \end{proof}
 

\section{Structure of the tropical Hodge bundle}\label{section_tropHodge}

Let $\Gamma$ be a tropical curve with a fixed minimal model $G$. As explained in \cite[Section 5.2]{AminiCaporaso_RiemannRoch}, the canonical divisor on $\Gamma$ is defined to be 
\begin{equation*}
K_\Gamma=K_G=\sum_{v\in V(G)} (2h(v) +\vert v\vert -2)(v) \ , 
\end{equation*}
where $\vert v\vert$ denotes the valence of the vertex $v$. 
Observe that $\deg(K_\Gamma)=2g-2$. The $h(v)$-term in the sum should hereby be thought of as contributing  $h(v)$ infinitely small loops at the vertex $v$. In fact, given a semistable curve $C$ whose dual graph is $G$, the canonical divisor is the multidegree of the dualizing sheaf on $C$ (see \cite[Remark 3.1]{AminiBaker_metrizedcurvecomplexes}). We recall Definition \ref{def_tropHodge} from the introduction. 

\begin{definition}
	Let $g\geq 2$. As a set, the \emph{tropical Hodge bundle} $\Lambda_g^{trop}$ is defined to be
	\begin{equation*}
	\Lambda_g^{trop}=\big\{([\Gamma], f)\big\vert [\Gamma]\in M_g^{trop} \textrm{ and } f\in\Rat(\Gamma) \textrm{ such that }K_\Gamma+(f)\geq 0\big\} 
	\end{equation*}
	and the projective tropical Hodge bundle $\calH_g^{trop}$ as
	\begin{equation*}
	\calH_g^{trop}=\big\{([\Gamma], D)\big\vert [\Gamma]\in M_g^{trop} \textrm{ and } D\in \vert K_\Gamma \vert\big\} \
	\end{equation*}
\end{definition}

The tropical Hodge bundles come with natural projection maps
\begin{equation*}
\Lambda_g^{trop}\longrightarrow M_g^{trop}\qquad \textrm{and} \qquad \calH_g\longrightarrow M_g^{trop}
\end{equation*}
given by $\big([\Gamma],f\big)\mapsto [\Gamma]$ and $\big([\Gamma],D\big)\mapsto [\Gamma]$, which, in abuse of notation, we both denote by $\pi_g$. 

In order to understand the structure of the tropical Hodge bundle $\Lambda_g^{trop}$ we consider the pullback of $\Lambda_g^{trop}$ and $\calH_g^{trop}$ to $\widetilde{M}_G$, defined as 
	\begin{equation*}
\widetilde{\Lambda}_G=\big\{([\Gamma], f)\big\vert [\Gamma]\in \widetilde{M}_G \textrm{ and } f\in \Rat(\Gamma) \textrm{ such that } K_\Gamma + (f)\geq 0\big\} \,
\end{equation*}
and
	\begin{equation*}
	\widetilde{\calH}_G=\big\{([\Gamma], D)\big\vert [\Gamma]\in \widetilde{M}_G \textrm{ and } D\in\vert K_\Gamma\vert \big\} \ .
	\end{equation*} 
In analogy with the space $S(D)$, as in Section \ref{section_linearsystems} above, we also set
	\begin{equation*}\begin{split}
	\widetilde{S}_G=\big\{([\Gamma], f, p_1, \ldots, p_{2g-2})\big\vert &[\Gamma]\in \widetilde{M}_G \textrm{ and } f\in\Rat(\Gamma) \\ &\textrm{ such that } K_\Gamma+(f)=p_1 +\ldots +p_{2g-2}\geq 0\big\} \ .
	\end{split}
	\end{equation*}

\begin{proposition}\label{prop_quotients}\begin{enumerate}[(i)]
\item The action of $S_{2g-2}$ on $S_G$ that permutes the points $p_1,\ldots, p_{2g-2}$ induces a natural bijection
\begin{equation*}
\widetilde{\Lambda}_G\simeq \widetilde{S}_G/S_{2g-2} \ .
\end{equation*}
\item The action of the additive group $\RR=(\RR,+)$ on $\widetilde{\Lambda}_G$, given by adding constant functions to $f$, induces a natural bijection  
\begin{equation*}
\widetilde{\calH}_G \simeq \widetilde{\Lambda}_G/\RR \ .
\end{equation*}
\end{enumerate}
\end{proposition}

\begin{proof}
The projections $\widetilde{S}_G\rightarrow \widetilde{M}_G$ and $\widetilde{\Lambda}_G\rightarrow\widetilde{M}_G$ are both invariant under the action of $S_{2g-2}$ and $\RR$. Therefore our claims follow from the respective identities on the fibers. \qed
\end{proof}

Let us now recall Theorem \ref{thm_main} from the introduction. 
\begin{theorem} Let $g\geq 2$. 
\begin{enumerate}[(i)]
\item The tropical Hodge bundles $\Lambda_g^{trop}$ and $\calH^{trop}_g$ canonically carry the structure of a generalized cone complex.
\item The dimensions of $\Lambda_g^{trop}$ and $\calH_g^{trop}$ are given by
\begin{equation*}
\dim\Lambda_g^{trop}=5g-4\qquad \textrm{ and } \qquad \dim\calH_g^{trop}=5g-5
\end{equation*}
respectively.
\item There is a proper subdivision of $M_g^{trop}$ such that, for all $[\Gamma]$ in the relative interior of a cone in this subdivision, the canonical linear systems $\vert K_\Gamma\vert=\pi_g^{-1}\big([\Gamma]\big)$ have the same combinatorial type. 
\end{enumerate}
\end{theorem}

\begin{proof}[Proof of Theorem \ref{thm_main}]
Part (i):
We are going to show that $\widetilde{S}_G$ canonically carries the structure of a cone complex. Then, by Proposition \ref{prop_quotients} above, both $\widetilde{\calH}_G$ and $\widetilde{\Lambda}_G$ carry the structure of a generalized cone complex. 

Choose an orientation for each edge $e$ of $G$, identifying it with the closed interval $[0,l(e)]$. As in Proposition \ref{prop_polyhedral} above, we can describe the cells of $\widetilde{S}_G$ by the following discrete data:
\begin{enumerate}[(i)]
\item a partition of $\{p_1,\ldots, p_{2g-2}\}$ into disjoint subsets $P_e$ and $P_v$ (indexed by vertices $v\in V(G)$ and edges $e\in E(G)$) that tells us on which edge (or at which vertex) each $p_i$ is located,
\item a total order on each $P_e$, and
\item the integer slope $m_e$ of $f$ at the starting point of $e$
\end{enumerate} 
such that for every vertex $v$ the equality
\begin{equation*}
d_v=2h(v)-2+\vert v\vert+ \sum_{\textrm{outward edges at } v}m_e +\sum_{\textrm{inward edges at } v} -(d_e+m_e)  
\end{equation*}
holds, where $d_v=\# P_v$ and $d_e=\# P_e$. The continuous parameters describing all elements in our cell are given by
\begin{enumerate}[(i)]
\item the values $f(v)$, 
\item the distances $d(p_{i}^{e})$ of $p_{i}^{e}$ from $0\in [0,l(e)]$, and
\item the lengths $l(e)$. 
\end{enumerate}

In order to find the conditions on those parameters, we again write $\sum_{p\in P_e} p=\sum d_{e,j}x_j$ for $x_1<\cdots< x_r$. Using this notation we have $0<x_1<\ldots <x_r<l(e)$ as conditions on the $d(p_{i}^{e})=x_i$ as well as by the continuity of $f$:
\begin{equation*}\begin{split}
m_{e}x_1&=f(x_1)-f(v) \\
(m_e+d_{e,1})(x_2-x_1)&=f(x_2)-f(x_1)\\
(m_e+d_{e,1}+d_{e,2})(x_3-x_2)&=f(x_3)-f(x_2)\\
&\vdots\\
\big(m_e+\sum_{j=1}^r d_{e,j}\big)(l(e)-x_r)&=f(v')-f(x_r) \ .
\end{split}\end{equation*}
Eliminating the non-parameters $f(x_1),\ldots, f(x_r)$ we can combine the system of equations to
\begin{equation}\label{equation_edge}
f(v')=f(v) +(m_e+d_e)l(e) - \sum_{j=1}^{r}d_{e,j}x_j \ .
\end{equation}
Since these conditions are invariant under multiplying all parameters simultaneously by elements in $\RR_{\geq 0}$, every non-empty cell in $\widetilde{S}_G$ has the structure of a rational polyhedral cone.  

Finally, the natural action of $\Aut(G)$ on $\widetilde{S}_G$, given by 
\begin{equation*}
\phi \cdot \big([\Gamma],f,p_1,\ldots,p_{2g-2}\big) = \big([\phi(\Gamma)],f\circ \phi^{-1},\phi(p_1),\ldots \phi(p_{2g-2})\big) 
\end{equation*}
for $\phi\in\Aut(G)$ is compatible with both the $S_{2g-2}$- and the $\RR$-operation. Moreover, given a weighted edge contraction $G'=G/e$ of $G$, the natural map $\widetilde{S}_{G'}\hookrightarrow \widetilde{S}_G$ identifies $\widetilde{S}_{G'}$ with the subcomplex of $\widetilde{S}_G$ given by the condition $l(e)=0$ in the above coordinates. 

Therefore we can conclude that both
\begin{equation*}
\Lambda_g^{trop}=\lim_{\longrightarrow}\widetilde{\Lambda}_G \qquad \textrm{ and }\qquad \calH_g^{trop}=\lim_{\longrightarrow} \widetilde{\calH}_G \ ,
\end{equation*}
where the limits are taken over the category $J_g$ as in Section \ref{section_tropMg} above, carry the structure of a generalized cone complex. 

Part (ii):
We need to show that the dimension of a maximal-dimensional cone in $\calH_g$ is $5g-5$. By \cite[Proposition 3.2.5 (i)]{BrannettiMeloViviani_tropicalTorelli}, we have $\dim M_g^{trop}=3g-3$ and, by \cite[Corollary 7]{Lin_linearsystems}, the dimension of the fiber $|K_{\Gamma}|$ of a point $[\Gamma]$ is at most $\deg(K_{\Gamma})=2g-2$. This shows that the dimension of $\calH_g^{trop}$ is at most $(3g-3)+(2g-2)=5g-5$. 

In addition we now exhibit a $(5g-5)$-dimensional cone in $\calH_g^{trop}$ as follows: Consider the tropical curve $\Gamma_{\max}$ as indicated in Figure \ref{figure_gammamax} and note that it has $2g-2$ vertices and $3g-3$ edges.
\begin{figure}[h]
   		\centering
   		\includegraphics[scale=0.4]{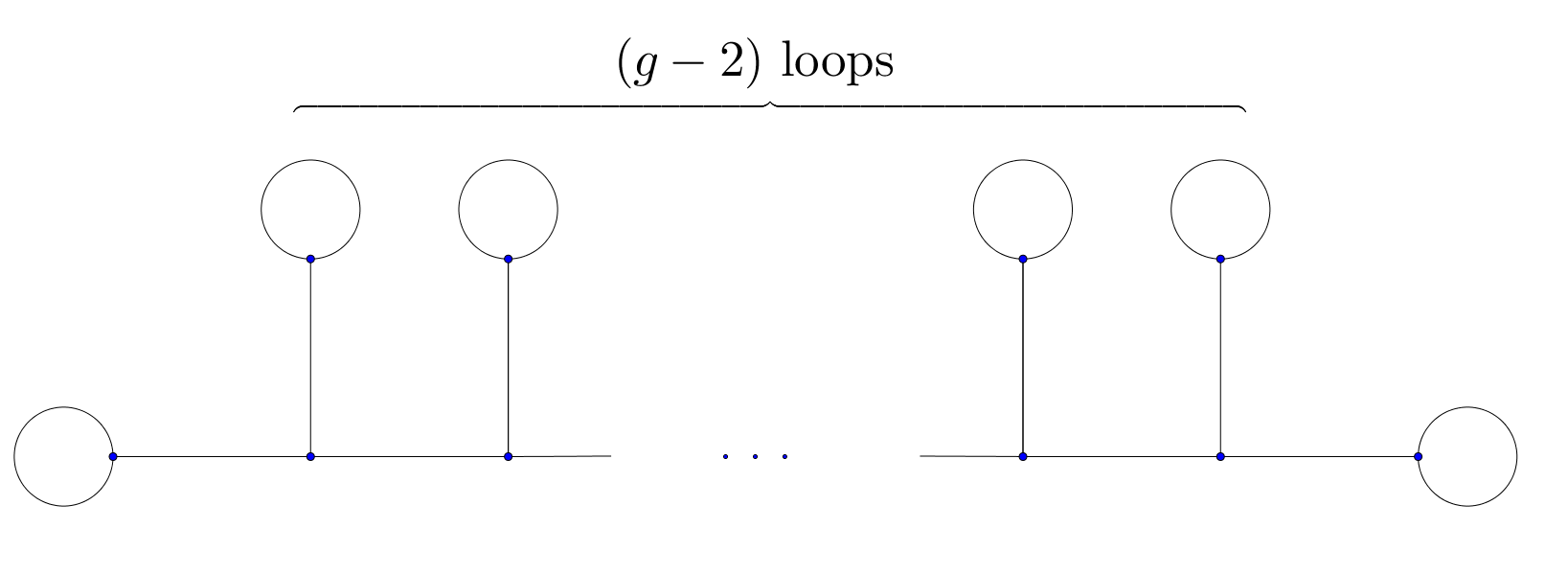}
   		\caption{The tropical curve $\Gamma_{\max}$ with $2g-2$ vertices (bold) and $3g-3$ edges.}\label{figure_gammamax}
   	\end{figure}

\begin{lemma}\label{lemma_edgelength}
Let $\Gamma$ be a tropical curve with minimal model $G=(V,E)$. Let $D$ be a divisor on $\Gamma$ such that the support of $D$ is contained in $V$. Then the combinatorial structure of $|D|$ is independent of the length of any loop or bridge in $G$.
\end{lemma}
	
A proof of Lemma \ref{lemma_edgelength} is provided below. Lemma \ref{lemma_edgelength} implies that the combinatorial structure of $|K_{\Gamma_{\max}}|$ is independent of the edge lengths, so we can choose a generic chamber. We obtain a divisor $D\in |K_{\Gamma_{\max}}|$ as indicated in Figure \ref{figure_divisormax}.  
   	
   	\begin{figure}[h]
   		\centering
   		\includegraphics[scale=0.42]{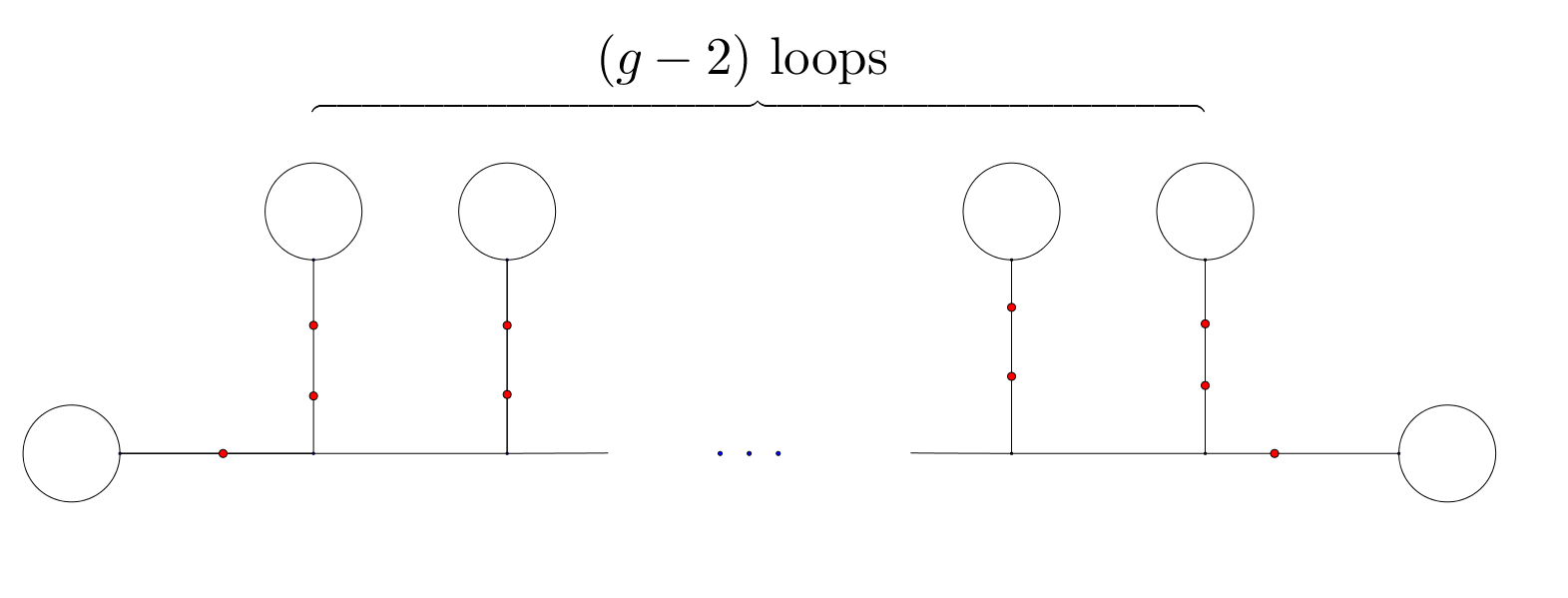}
   		\caption{The divisor $D$ on $\Gamma_{\max}$ (red).}\label{figure_divisormax}
   	\end{figure}
   
   	By \cite[Proposition 13]{HaaseMusikerYu_linearsystems}, the divisor $D$ belongs to a $(2g-2)$-dimensional face in $|K_{\Gamma_{\max}}|$. Thus there is a $(5g-5)$-dimensional cone in $\calH_g^{trop}$.

Part (iii): We use the coordinates described in part (i) above. For every edge $e$ of $G$ with $d_e=\# P_e=0$ we have an equation $m_e l(e)=f(v)-f(v')$ which is parametrized by the $l(e)$. So suppose $d_e>0$ and set $x=\frac{1}{d_e}\sum_{j=1}^r d_{e,j} x_j$. Then (\ref{equation_edge}) can be rewritten as 
\begin{equation*}
f(v')=f(v) +(m_e+d_e)l(e) - d_{e}x \ .
\end{equation*}
Eliminating $x$ from this equation, subject to the condition $0<x<l(e)$ we obtain
\begin{equation*}\begin{split}
(-d_e-m_e)l(e)&>f(v)-f(v')\\
m_el(e)&<f(v')-f(v) 
\end{split}\end{equation*}
and therefore we obtain that the images of cells in $\widetilde{M}_G$ are polyhedra. Moreover, the combinatorial type of $\vert K_\Gamma\vert$ is independent under scaling all edge lengths with a factor in $\RR_{>0}$ and thus all these polyhedra determine a subdivision of $\widetilde{M}_G$ such that on each relatively open cell of this subdivision, the corresponding $\vert K_\Gamma\vert$ has the same set of cells, i.e., the combinatorial type of $\widetilde{S}_G$ is constant. \qed

\end{proof}

    
    

\begin{proof}[of Lemma \ref{lemma_edgelength}]
	Suppose $e_{b}\in E$ is a bridge in $G$. Let $\Gamma_1$ and $\Gamma_2$ be two tropical curves with minimal model $G$ such that the length of $e$ in $\Gamma_1$ and $\Gamma_2$ is $l$ and $c\cdot l$ respectively (where $l,c>0$), and for $e\in E-\{e_b\}$, the lengths of $e$ in $\Gamma_1$ and $\Gamma_2$ are the same. It suffices to show that the sets of cells in $|D|_{\Gamma_1}$ and $|D|_{\Gamma_2}$ are exactly the same.
	
	In $\Gamma_1$, we view the bridge $e_b$ as the open interval $(0,l)$. For any cell $C_1$ in $|D|_{\Gamma_1}$, its data consists of an integer $m_{e_b}$ and a partition of nonnegative integers $d_{e_b}=\sum_{j=1}^{r}{d_{e,j}}$. Suppose a divisor $D+\Div(f_1)$ is $\sum_{j=1}^{r}{d_{e,j}\cdot x_j}$ on the bridge $e_b$, where $0<x_1<\cdots<x_{r(e)}<l$. Here the rational function $f_1$ is unique up to a translation. So we may assume that the value of $f_1$ is zero on the endpoint $0$ of $e_b$. Then on $e_b$ the function $f_1$ is defined inductively as follows:
\begin{itemize}
	\item For $0<x\leq x_1$ it is given by 
	\begin{equation*}
		f_1(x)=m_{e_b}\cdot x \ .
	\end{equation*}
	\item Given $x_k<x<x_{k+1}$ for $1\leq k\leq r(e)$ we have
	\begin{equation*}
		f_1(x)=f_1(x_k) + \big(m_{e_b}+\mathop{\sum}_{j=1}^{k}{d_{e_b,j}}\big)\cdot (x-x_{k})
	\end{equation*}
\end{itemize}

	
	Now on $\Gamma_2$ we also view the bridge $e_b$ as the open interval $(0,c\cdot l)$, with the same orientation. We construct a rational function $f_2$ on $\Gamma_2$. First, we define $f_2$ on the bridge $e_b$ inductively as follows:
\begin{itemize}
	\item For $0<x\leq cx_1$ we define $f_2(x)$ by 
	\begin{equation*}
		f_2(x)=m_{e_b}\cdot x \ ,
	\end{equation*}
	\item and, given $cx_k<x<cx_{k+1}$ for $1\leq k\leq r(e)$, we set
	\begin{equation*}
		f_2(x)=f_1(cx_k) + \big(m_{e_b}+\mathop{\sum}_{j=1}^{k}{d_{e_b,j}}\big)\cdot (x-cx_{k})
	\end{equation*}
\end{itemize}

	Since $e_b$ is a bridge in $G$, the graph $G-e_b$ consists of two connected components. We denote them by $G_1$ and $G_2$, where $G_1$ contains the endpoint $0$ of $e_b$ and $G_2$ contains the endpoint $cl$ of $e_b$. For convenience we let
	\[f_{1}(l) = m_{e_b}\cdot x_1+(m_{e_b}+d_{e_b,1})\cdot x_2+\cdots+(m_{e_b}+\mathop{\sum}_{j=1}^{r}{d_{e_b,j}})\cdot (l-x_{r(e)})\]
	and
	\[f_{2}(cl) = c\cdot f_{1}(l).\]
	Then we define $f_2$ on $G-e_b$ as follows:
	\begin{equation*}
	f_{2}(x)=\begin{cases}
	f_{1}(x), &\text{ if } x\in G_1;\\
	f_{1}(x)+f_{2}(cl)-f_{1}(l)=f_{1}(x)+(c-1)f_{1}(l), &\text{ if } x\in G_2.
	\end{cases}
	\end{equation*}
	By definition, $f_1$ and $f_2$ admit the same data on $G_i$ for $i=1,2$. In addition, on the bridge $e_b$, both functions admit the integer $m_{e_b}$ and the same partition $\sum_{j=1}^{r}{d_{e_b,j}}$. So $f_2$ corresponds to a cell $C_2$ in $|D|_{\Gamma_2}$ that is exactly the same as $C_1$. For the same reason we can get the cell $C_1$ from $C_2$ (just note that $l=(\frac{1}{c})\cdot (cl)$). So Lemma \ref{lemma_edgelength} holds for bridges.
	
	Suppose $e_{l}$ is a loop in $G$. In this case almost the same proof works, except that the condition $f_{1}(l)=0$ must hold, and $G'=G-e_{l}$ is connected. We have $f_{2}(cl)=0$; therefore we may define $f_2$ on $e_{l}$ in the same way as on $e_{b}$, as well as $f_{2}(x)=f_{1}(x)$ for all $x\in G'$. Thus our claim also holds for loops.\qed
\end{proof}


\section{Computations in low genus}\label{section_examples}

In this section we present some computational results on the polyhedral structure of tropical Hodge bundles of small genus. In order to describe all cones in $\Lambda_g^{trop}$ we first list all cones in $M_g^{trop}$. Then for each cone, we compute its subdivision by the structure of $|K_{\Gamma}|$. It turns out that already the two cases $g=2$ and $g=3$ show a surprisingly distinct behavior.

\begin{proposition}\label{proposition_invariantg2}
	Let $\Gamma$ be a tropical curve in $M_2^{trop}$. Then the combinatorial structure of $|K_{\Gamma}|$ is uniquely determined by the minimal model $G$ of $\Gamma$. In other words, it is independent of the edge lengths in $\Gamma$.
\end{proposition}

\begin{proof}
	There are $7$ faces in $M_2^{trop}$ as in \cite[Figure 4]{Chan_tropicalTorelli}. For $6$ faces among them, all edges are loops or bridges, so the claim follows from Lemma \ref{lemma_edgelength}. For the "theta graph" $G_{\theta}$, by explicit computation we know that the canonical linear system $|K_{G_{\theta}}|$ is always a one-dimensional polyhedral complex with three segments, as in Figure \ref{figure_theta}.\qed
	\begin{figure}[h]
		\centering
		\includegraphics[scale=0.5]{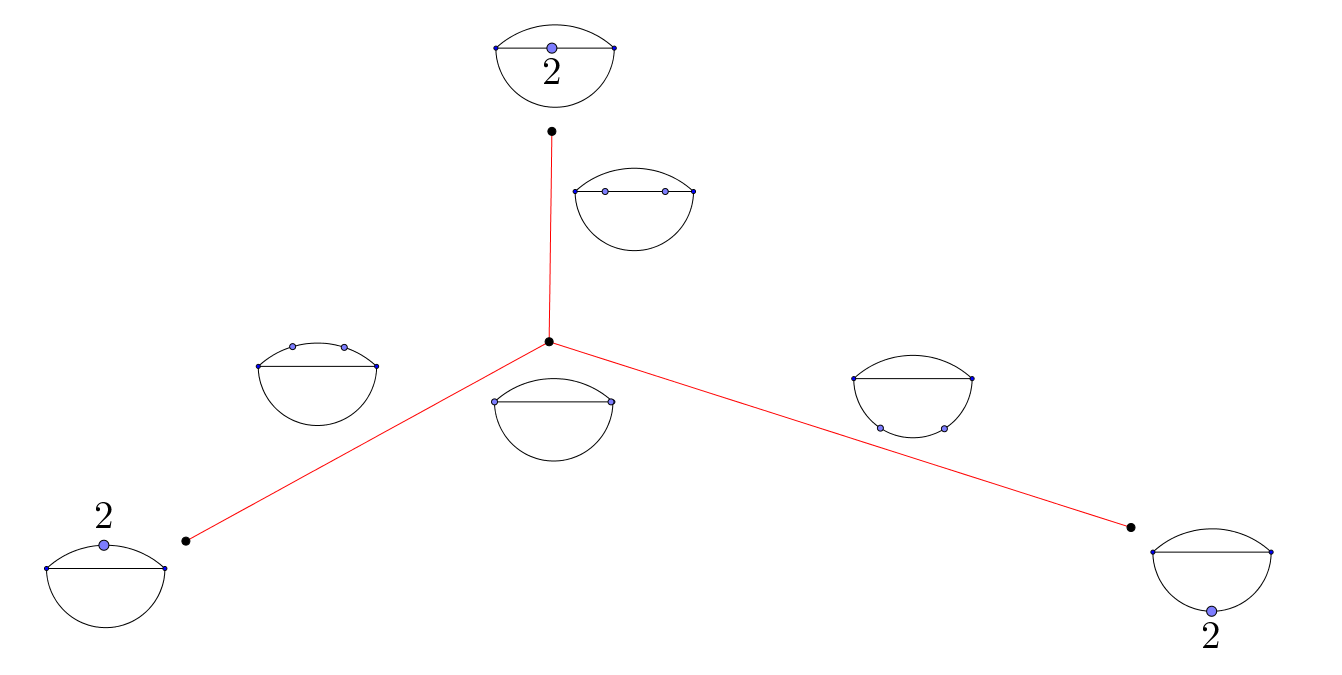}
		\caption{The polyhedral complex $|K_{G_{\theta}}|$. Its $f$-vector is $(4,3)$.}\label{figure_theta}
	\end{figure}
\end{proof}

The face lattice of $\Lambda_2^{trop}$ is visualized in Figure \ref{figure_lambda2} from the introduction.

\begin{remark}\label{remark_dumbbell}
	The $f$-vector of $\Lambda_2^{trop}$ is $(1,5,11,16,9,1)$, which is consistent with Theorem \ref{thm_main} (ii). The unique $5$-dimensional face consists of the "dumbbell" graph and a triangular cell in $|K_{\Gamma}|$. In other words, any divisor in this cell is of the form $P+Q$, where $P$ and $Q$ are distinct points in the interior of the bridge in the dumbbell graph.
\end{remark}

When $g=3$, the counterpart of Proposition \ref{proposition_invariantg2} is no longer true. One counterexample consists of the $6$-dimensional cone $C\simeq \mathbb{R}_{>0}^{6}$ in $M_3^{trop}$ parametrizing tropical curves whose minimal model $G$ is a complete metric graph $K_4$. The following proposition characterizes the open chambers of $C$ regarding the structure of $|K_{G}|$. It is a result of explicit computations using the algorithm described in \cite[Section 2.3]{Lin_linearsystems}.

\begin{proposition}\label{proposition_51chambers} 
	\begin{enumerate}[(i)]
\item There are $51$ open chambers in $C$. For all metrics in the same chamber, the canonical linear system $|K_{G}|$ has the same set of cells. In that case, the polyhedral complex $|K_{G}|$ always has $34$ vertices, $60$ edges, and $27$ two-dimensional faces ($12$ triangles and $15$ quadrilaterals). However, there are $4$ non-isomorphic combinatorial structures of $|K_{G}|$.
		\item Let $M=(M_{12},M_{13},M_{14},M_{23},M_{24},M_{34})$ be a metric. Consider the following four $3$-subsets:
		\begin{equation}\label{equation_triples}
		\{M_{12},M_{13},M_{14}\},\{M_{12},M_{23},M_{24}\},\{M_{13},M_{23},M_{34}\},\{M_{14},M_{24},M_{34}\}.
		\end{equation}
		Then $M$ belongs to an open chamber if and only if among the elements of each $3$-subset in (\ref{equation_triples}), the minimum is attained only once.
	\end{enumerate}
\end{proposition}

\begin{remark}[The structure of $|K_{G}|$ for a generic metric]
	If $M$ belongs to an open chamber, the canonical linear system $|K_{G}|$ always has the $13$ vertices in Figure \ref{figure_link}.
	\begin{figure}[h]
		\begin{minipage}[c]{0.45\textwidth}
		\includegraphics[scale=0.5]{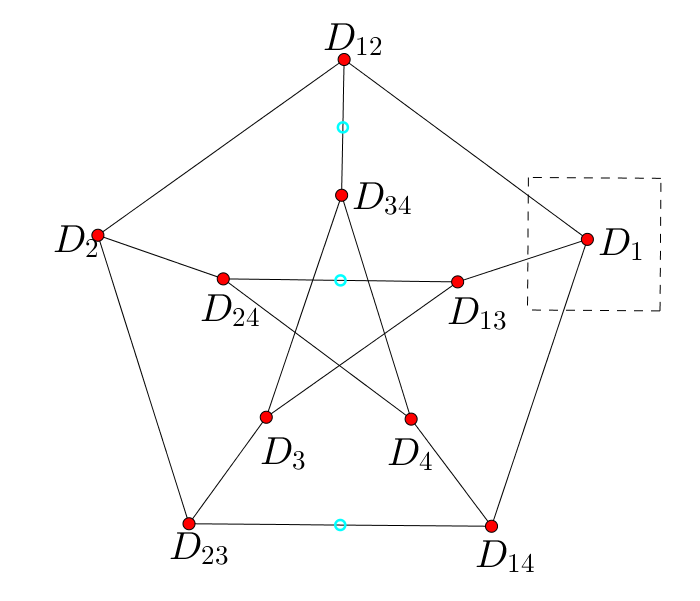}
		\caption{The main skeleton of $|K_{G}|$. An extra vertex is connected to the $10$ vertices in red.}\label{figure_link}
		\end{minipage}
		\hfill
		\begin{minipage}[c]{0.45\textwidth}
			\includegraphics[scale=0.6]{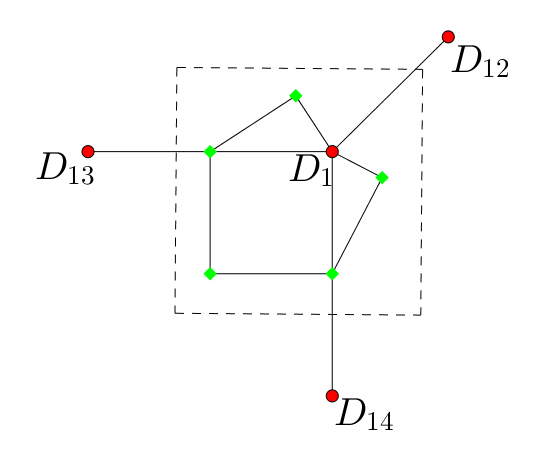}
			\caption{Each bat consists of $5$ extra vertices (green). This bat appears when $M_{12}<M_{13},M_{14}$.}\label{figure_bat}
		\end{minipage}
	\end{figure}
	Among them, the $10$ labeled vertices are all connected to an extra vertex that is the divisor $K_{G}$. The remaining $20$ vertices come from $4$ copies of a sub-structure (we call a \emph{bat}) attached at $D_1$, $D_2$, $D_3$, and $D_4$. Note that some edges in Figure \ref{figure_link} are subdivided by other vertices in the bats.

	The $4$ distinct combinatorial types of $|K_{G}|$ come from different ways of attaching the bats. Since $M$ belongs to an open chamber, the minimum of $M_{12},M_{13},M_{14}$ appears only once. Suppose it is $M_{12}$, then the bat at $D_1$ is attached along the edges towards $D_{13}$ and $D_{14}$, as in Figure \ref{figure_bat}.
Figure \ref{figure_divisors} shows the divisors $D_{i}$ and $D_{ij}$.
	\begin{figure}[H]
		\begin{minipage}[c]{0.45\textwidth}
			\centering
			\includegraphics[scale=0.4]{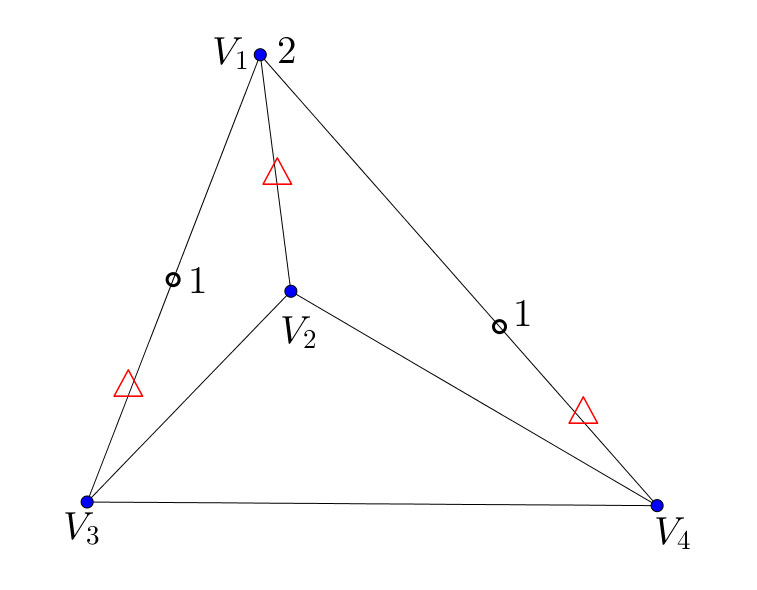}
			\centerline{The divisor $D_{1}$.}
		\end{minipage}
		\hfill
		\begin{minipage}[c]{0.45\textwidth}
			\centering
			\includegraphics[scale=0.4]{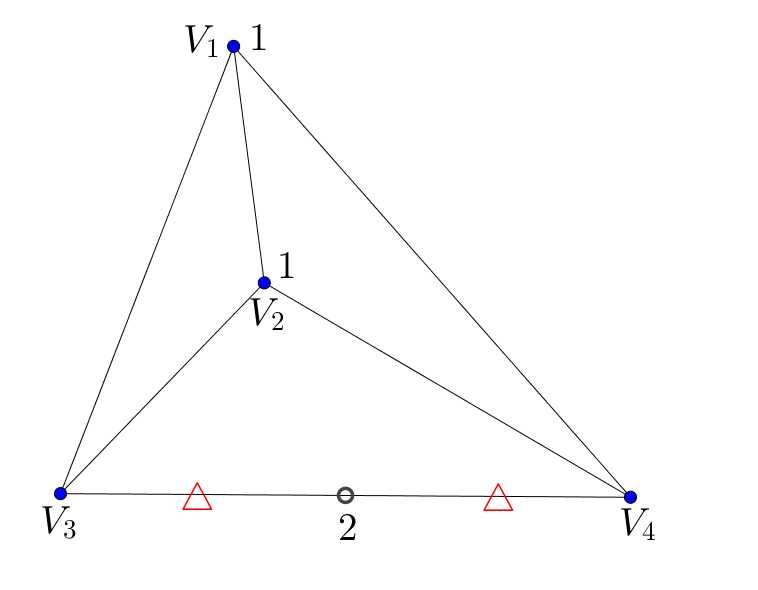}
			\centerline{The divisor $D_{34}$.}
		\end{minipage}
		\caption{Numbers are the coefficients of divisors and the $\triangle$ symbols show the equal line segments.}\label{figure_divisors}
	\end{figure}
\end{remark}

\begin{remark}[The boundary of the $51$ open chambers]\label{remark_walls}
	The action of the permutation group $S_4$ on the vertices of $K_4$ induces $4$ orbits among the $51$ open chambers, with lengths $24$, $12$, $12$, and $3$. Each orbit corresponds to a combinatorial type of $|K_{G}|$. Each open chamber is an open cone in $C$, defined by homogeneous linear inequalities involving $M_{12}$, $M_{13}$, $M_{14}$, $M_{23}$, $M_{24}$, and $M_{34}$. The inequalities are displayed as the \emph{covers} in a lattice. For example, $M_{13}$ covering $M_{12}$ means that the inequality $M_{13}>M_{12}$ holds.
\end{remark}

\begin{figure}[H]
	\begin{minipage}[c]{0.45\textwidth}
		\centering
		$\xymatrix @!0 { 
			M_{24}\ar@{-}[rd] & & M_{34}\ar@{-}[ld]\ar@{-}[rd] & & M_{23}\ar@{-}[ld] \\
			& M_{14}\ar@{-}[rd] & & M_{13}\ar@{-}[ld] & \\
			& & M_{12} & & }$
	\end{minipage}
	\hfill
	\begin{minipage}[c]{0.45\textwidth}
		\centering
		$\xymatrix @!0 { 
			M_{23}\ar@{-}[rd] & & M_{34}\ar@{-}[ld]\ar@{-}[rd] & & M_{14}\ar@{-}[ld] \\
			& M_{13}\ar@{-}[rd] & & M_{24}\ar@{-}[ld] & \\
			& & M_{12} & & }$
	\end{minipage}
	\caption{Representatives of two chamber orbits of length $12$.}
\end{figure}
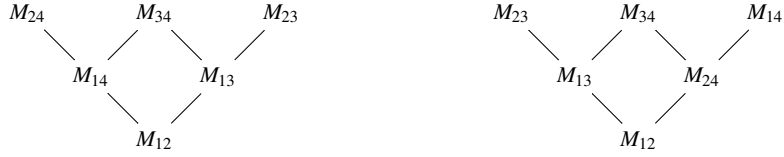

\begin{figure}[H]
	\begin{minipage}[c]{0.45\textwidth}
		\centering
		$\xymatrix @!0 { 
			M_{13}\ar@{-}[rd]\ar@{-}[rrd] & M_{14}\ar@{-}[d]\ar@{-}[rd] & M_{23}\ar@{-}[d]\ar@{-}[ld] & M_{24}\ar@{-}[ld]\ar@{-}[lld] \\
			& M_{12} & M_{13} & }$
		\caption{Representatives of a chamber orbit of length $3$.}
	\end{minipage}
	\hfill
	\begin{minipage}[c]{0.45\textwidth}
		\centering
		$\xymatrix @!0 { 
			M_{14}\ar@{-}[rd] & & M_{24}\ar@{-}[ld] & \\
			& M_{34}\ar@{-}[rd] & & M_{23}\ar@{-}[ld] \\
			& & M_{13}\ar@{-}[d] & \\
			& & M_{12} &
		}$
		\caption{Representatives of a chamber orbit of length $24$.}
	\end{minipage}
\end{figure}

\section{The realizability problem}\label{section_realizability}

Let $k$ be an algebraically closed field carrying the trivial absolute value. In \cite{AbramovichCaporasoPayne_tropicalmoduli}, expanding on earlier work (see e.g. \cite{Viviani_tropvscompTorelli}), the authors have constructed a natural (continuous) tropicalization map 
\begin{equation*}\begin{split}
\trop_{\calM_g}\mathrel{\mathop:}\calM_g^{an}&\longrightarrow M_g^{trop} \\
x&\longmapsto [\Gamma_x]
\end{split}\end{equation*}
from the non-Archimedean analytic moduli space $\calM_g^{an}$ onto $M_g^{trop}$. Let us recall the construction of $\Gamma_x$: A point $x\in \calM_g^{an}$ parametrizes an algebraic curve $C$ over some non-Archimedean extension $K$ of $k$. Possibly after a finite extension $K'\vert K$ we can extend $C$ to a stable model $\calC\rightarrow \Spec R'$ over the valuation ring $R'$ of $K'$. Denote by $G_x$ the weighted dual graph of the special fiber $\calC_s$ of $\calC$, whose vertices correspond to the components of $\calC_s$ and in which we have an edge $e$ between two vertices $v$ and $v'$ for every node connecting the two corresponding components $C_v$ and $C_{v'}$. The vertex weight function is given by
\begin{equation*}
h(v)=g(\widetilde{C}_v) \ ,
\end{equation*}
where $\widetilde{C}_v$ denotes the normalization of $C_v$. Finally, around every node $p_e$ in $\calC_s$ we can find formal coordinates $x$ and $y$ of $\calC$ such that $xy=t$ for some element $t$ in the base. Then the edge length of $e$ is given by $l(e)=\val(t)$. 

Denote by $\calH_g^{an}$ the non-Archimedean analytification of the total space of the algebraic Hodge bundle $\calH_g$. 

\begin{proposition}
There is a natural tropicalization map $\trop_{\calH_g}\mathrel{\mathop:}\calH_g^{an}\rightarrow \calH_g^{trop}$ that makes the diagram
\begin{equation*}
\begin{CD}
\calH_g^{an}@>\trop_{\calH_g}>> \calH_g^{trop}\\
@VVV @VVV\\
\calM_g^{an}@>\trop_{\calM_g}>>M_g^{trop}
\end{CD}
\end{equation*}
commute. 
\end{proposition}

We expect that $\trop_{\calH_g}$ is also continuous, but refrain from investigating this question here, since such an investigation appears to be too technical for the nature of a contribution to this volume. 

\begin{proof}
An element $x\in \calH_g^{an}$ parametrizes a tuple $(C,K_C)$ consisting of a smooth projective curve $C$ over a non-Archimedean extension $K$ of $k$ together with a canonical divisor on $C$. Then we may associate to $(C,K_C)$ the point $\big([\Gamma_x],\tau_\ast(K_C)\big)$, where 
\begin{equation*}
\tau_\ast\mathrel{\mathop:}\Div(C_{\overline{K}})\longrightarrow \Div(\Gamma)
\end{equation*} 
denotes the specialization map constructed in \cite[Section 2.3]{Baker_specialization} that is given by pushing $K_C$ forward to the non-Archimedean skeleton of $\calC$. As shown in \cite[Section 2.3]{Baker_specialization} and the references therein, this is well-defined and the commutativity of the above diagram is an immediate consequence of the definition. \qed
\end{proof}

It is well-known that $\trop_g\mathrel{\mathop:}\calM_g^{an}\rightarrow M_g^{trop}$ is surjective. By Theorem \ref{thm_main} (ii) we have that
\begin{equation*}
\dim_{\mathbb{C}}\calH_g=4g-4<5g-5=\dim\calH_g^{trop}
\end{equation*}
and therefore the analogous statement for $\calH_g$ appears to be false. This gives rise to the following problem.

\begin{problem}\label{problem_realizability}
Find a characterization of the locus $\trop_{\calH_g}(\calH_g^{an})$ in $\calH_g^{trop}$, the so-called \emph{realizability locus} in $\calH_g^{trop}$. 
\end{problem}

In other words, given a (stable) tropical curve $\Gamma$ of genus $g$ together with a divisor $D$ that is equivalent to $K_\Gamma$, find algebraic and combinatorial conditions that ensure that there is an algebraic curve $C$ over a non-Archimedean field extension $K$ of $k$ together with a canonical divisor $\widetilde{D}$ on $C$ such that 
\begin{equation*}
\trop_{\calH_g}\big([C],\widetilde{D}\big)=\big([\Gamma],D\big) \ .
\end{equation*}

Since $\trop_{\calM_g}$ is surjective, we know already that every tropical curve $\Gamma$ can be lifted to a smooth algebraic curve $C$. In the case of $\Gamma$ having integer edge lengths $l(e)$ we can give a constructive approach to this problem: Consider a special fiber $\calC_s$ over $k$ whose weighted dual graph is $G$, then apply logarithmically smooth deformation theory to find a smoothing of $\calC_s$ to a stable family $\calC\rightarrow \Spec R$ with deformation parameters $l(e)$ at each node (see e.g. \cite[Proposition 3.38]{Gross_book}). If all $l(e)=1$, we may alternatively also proceed as in \cite[Appendix B]{Baker_specialization}. 

Now let $\widetilde{D}$ be a divisor on $C$ that specializes to the given canonical divisor $D$ on $\Gamma$. Since $\deg \widetilde{D}=\deg D=2g-2$, Clifford's theorem (or alternatively Baker's Specialization Lemma \cite[Corollary 2.11]{Baker_specialization}) shows that the rank of $\widetilde{D}$ is at most $g-1$. If the rank of $D$ is smaller than $g-1$ it cannot be a canonical divisor. If, however, the divisor $\widetilde{D}$ has rank $g-1$, then, by Riemann-Roch, it is a canonical divisor. So the realizability problem reduces to finding a lift of the divisor $D$ of rank  $g-1$. 

The existence of such a divisor would follow, for example, from the smoothness of a suitable moduli space of limit linear series (see e.g. \cite{EisenbudHarris_limitlinearseries} and \cite{Osserman_limitlinearseries}). Unfortunately the machinery of limit linear series is not available in full generality (i.e. for nodal special fibers that are not of compact type). However, considerations undertaken from the point of view of compactifications of the moduli space of abelian differentials and its strata (see in particular \cite{BainbridgeChenGendronGrushevskyMoeller_abeliandifferentials}) treating the special case of limits of canonical linear systems seem to provide us with a very promising approach for future investigations into this question.




\begin{acknowledgement}
This article was initiated during the Apprenticeship Weeks (22 August-2 September 2016), led by Bernd Sturmfels, as part of the Combinatorial Algebraic Geometry Semester at the Fields Institute. Both authors would like to acknowledge his input. Thanks are also due to the Max-Planck-Institute of Mathematics in the Sciences in Leipzig, Germany, for its hospitality. The second author, in particular, would like to thank Diane Maclagan for several discussion related to the topic of this note, as well as the Fields Institute for Research in Mathematical Sciences. Finally, many thanks are due to the anonymous referees for several helpful comments and suggestions.  
\end{acknowledgement}


\begin{thebibliography}{99.}%

\bibitem{AbramovichCaporasoPayne_tropicalmoduli} Abramovich, Dan; Caporaso, Lucia; Payne, Sam: \emph{The tropicalization of the moduli space of curves}, Ann. Sci. \'Ec. Norm. Sup\'er. (4) 48 (2015), no. 4, 765-809.

\bibitem{AminiBaker_metrizedcurvecomplexes}
Amini, Omid; Baker, Matthew: \emph{Linear series on metrized complexes of
  algebraic curves}, Math. Ann. \textbf{362} (2015), no.~1-2, 55--106.
  

\bibitem{AminiCaporaso_RiemannRoch} Amini, Omid; Caporaso, Lucia: \emph{Riemann-Roch theory for weighted graphs and tropical curves}, Adv. Math. 240 (2013), 1-23.

\bibitem{BainbridgeChenGendronGrushevskyMoeller_abeliandifferentials} Bainbridge, Matt; Chen, Dawei; Gendron, Quentin; Grushevsky, Samuel; M\"oller, Martin: \emph{Compactifications of strata of abelian differentials}, arXiv:1604:08834 [math] (2016).

\bibitem{Baker_specialization} Matthew Baker, \emph{Specialization of linear systems from curves to graphs},
  Algebra Number Theory \textbf{2} (2008), no.~6, 613--653, With an appendix by Brian Conrad.

\bibitem{BakerNorine_RiemannRoch} Baker, Matthew; Norine, Serguei: \emph{Riemann-Roch and Abel-Jacobi on a finite graph}, Adv. Math. 215 (2007), no. 2, 766-788. 

\bibitem{BrannettiMeloViviani_tropicalTorelli}
Brannetti, Silvia; Melo, Margarida; Viviani, Filippo: \emph{On the tropical
  {T}orelli map}, Adv. Math. \textbf{226} (2011), no.~3, 2546--2586.

\bibitem{Caporaso_tropicalmoduli}
Caporaso, Lucia: \emph{Algebraic and tropical curves: comparing their moduli
  spaces}, Handbook of moduli. {V}ol. {I}, Adv. Lect. Math. (ALM), vol.~24,
  Int. Press, Somerville, MA, 2013, pp.~119--160.

\bibitem{Caporaso_tropicalmoduliII} 
Caporaso, Lucia: \emph{Tropical methods in the moduli theory of algebraic curves}, arXiv:1606.00323 [math] (2016).

\bibitem{CaporasoViviani_tropicalTorelli}
Caporaso, Lucia; Viviani, Filippo: \emph{Torelli theorem for graphs and
  tropical curves}, Duke Math. J. \textbf{153} (2010), no.~1, 129--171.
  
\bibitem{CavalieriHampeMarkwigRanganathan_tropicalHassett}
Cavalieri, Renzo; Hampe, Simon; Markwig, Hannah; Ranganathan, Dhruv:
  \emph{Moduli spaces of rational weighted stable curves and tropical
  geometry}, Forum Math. Sigma \textbf{4} (2016), e9, 35.
  
\bibitem{CavalieriMarkwigRanganathan_tropicalHurwitz}
Cavalieri, Renzo; Markwig, Hannah; Ranganathan, Dhruv: \emph{Tropicalizing the
  space of admissible covers}, Math. Ann. \textbf{364} (2016), no.~3-4,
  1275--1313.

\bibitem{Chan_tropicalTorelli} Chan, Melody: \emph{Combinatorics of the tropical Torelli map}, Algebra Number Theory 6 (2012), no. 6, 1133-1169. 

\bibitem{Chan_topologyM2n}
Chan, Melody: \emph{Topology of the tropical moduli spaces $M_{2,n}$}, arXiv:1507.03878 [math] (2015).

\bibitem{ChanGalatiusPayne_tropicalmoduliII}
Chan, Melody; Galatius, Soren; Payne, Sam: \emph{The tropicalization of the moduli space of curves {II}: {Topology} and applications}, arXiv:1604.03176 [math] (2016).
  
\bibitem{ChanMeloViviani_tropicalmoduli}
Chan, Melody; Melo, Margarida; Viviani, Filippo:  \emph{Tropical
  {T}eichm\"uller and {S}iegel spaces}, Algebraic and combinatorial aspects of
  tropical geometry, Contemp. Math., vol. 589, Amer. Math. Soc., Providence,
  RI, 2013, pp.~45--85.
  
  \bibitem{EisenbudHarris_limitlinearseries}
Eisenbud, David; Harris, Joe: \emph{Limit linear series: basic theory},
  Invent. Math. \textbf{85} (1986), no.~2, 337--371. 
 

\bibitem{GathmannKerber_RiemannRoch} Gathmann, Andreas; Kerber, Michael: \emph{A Riemann-Roch theorem in tropical geometry}, Math. Z. 259 (2008), no. 1, 217-230.

\bibitem{GathmannKerberMarkwig_tropicalfans}
Gathmann, Andreas; Kerber, Michael; Markwig, Hannah: \emph{Tropical fans and
  the moduli spaces of tropical curves}, Compos. Math. \textbf{145} (2009),
  no.~1, 173--195.
  
\bibitem{GathmannMarkwig_tropicalKontsevich}
Gathmann, Andreas; Markwig, Hannah: \emph{Kontsevich's formula and the {WDVV}
  equations in tropical geometry}, Adv. Math. \textbf{217} (2008), no.~2,
  537--560.
 
\bibitem{Gross_book}
Gross, Mark: \emph{Tropical geometry and mirror symmetry}, CBMS Regional Conference Series in Mathematics, 114, Amer. Math. Soc., Providence, RI, 2011. 

\bibitem{HaaseMusikerYu_linearsystems} Haase, Christian; Musiker, Gregg; Yu Josephine: \emph{Linear systems on tropical curves}, Math. Z. 270 (2012), no. 3-4, 1111-1140. 

\bibitem{Lin_linearsystems} Lin, Bo: \emph{Computing linear systems on metric graphs}, arXiv:1603.00547 [math] (2016).

\bibitem{Mikhalkin_ICM}
Mikhalkin, Grigory: \emph{Tropical geometry and its applications}, International
  {C}ongress of {M}athematicians. {V}ol. {II}, Eur. Math. Soc., Z\"urich, 2006,
  pp.~827--852.
 
 \bibitem{Mikhalkin_Gokova}
Mikhalkin, Grigory: \emph{Moduli spaces of rational tropical curves}, Proceedings of
  {G}\"okova {G}eometry-{T}opology {C}onference 2006, G\"okova
  Geometry/Topology Conference (GGT), G\"okova, 2007, pp.~39--51.

\bibitem{MikhalkinZharkov_tropicalJacobians} Mikhalkin, Grigory; Zharkov, Ilia: \emph{Tropical curves, their Jacobians and theta functions}, In: Curves and abelian varieties, 203-230, Contemp. Math., 465, Amer. Math. Soc., Providence, RI, 2008.

\bibitem{Osserman_limitlinearseries}
Brian Osserman, \emph{Limit linear series for curves not of compact type},
  arXiv:1406.6699 [math] (2014).

\bibitem{Ranganathan_ratcurtorvar&nonArch}
Ranganathan, Dhruv: \emph{Moduli of rational curves in toric varieties and
  non-{Archimedean} geometry}, arXiv:1506.03754 [math] (2015).

\bibitem{Ulirsch_functroplogsch}
Ulirsch, Martin: \emph{Functorial tropicalization of logarithmic schemes: The case of constant coefficients}, arXiv: 1310:6269 [math] (2013). 

\bibitem{Ulirsch_tropicalHassett}
Ulirsch, Martin: \emph{Tropical geometry of moduli spaces of weighted stable curves},
  J. Lond. Math. Soc. (2) \textbf{92} (2015), no.~2, 427--450.

\bibitem{Vakil_tautologicalring}
Vakil, Ravi: \emph{The moduli space of curves and its tautological ring}, Notices Amer. Math. Soc. 50 (2003), no. 6, 647-658. 

\bibitem{Viviani_tropvscompTorelli}
Viviani, Filippo: \emph{Tropicalizing vs. compactifying the {T}orelli morphism},
  Tropical and non-{A}rchimedean geometry, Contemp. Math., vol. 605, Amer.
  Math. Soc., Providence, RI, 2013, pp.~181--210.

\bibitem{Yu_tropstablemaps}
Yue Yu, Tony: \emph{Tropicalization of the moduli space of stable maps}, Math. Z.
  \textbf{281} (2015), no.~3-4, 1035--1059.

%
%

%
%
%
%
%

%
%
%
%
%
%
%
%
%
%
%
%
%
%
%
%
%
%
%
%
%
%
%
%
%
%
%
%
%
%
%
%
%
\end{thebibliography}
\end{document}